\documentclass{amsproc}
\usepackage{amsfonts}
\usepackage{amsmath}
\usepackage{amssymb}
\usepackage{graphicx}
\usepackage{geometry}
\usepackage{amsthm}
\usepackage{latexsym}
\usepackage{fancyhdr}
\usepackage{amscd}
\usepackage{lscape}
\usepackage{tikz}
\usepackage{float}
\usepackage{bm}
\usepackage{subfigure}
\usepackage{grffile}
\usepackage{array}
\usepackage{caption}
\usepackage{longtable}
\newtheorem{theorem}{Theorem}[section]
\newtheorem{lemma}[theorem]{Lemma}
\newtheorem{proposition}[theorem]{Proposition}
\theoremstyle{definition}
\newtheorem{definition}[theorem]{Definition}
\newtheorem{example}[theorem]{Example}

\newenvironment{claim}[1][Claim]{\begin{trivlist}
\item[\hskip \labelsep {\bfseries #1}]}{\end{trivlist}}

\theoremstyle{remark}

\numberwithin{equation}{section}

\begin{document}

\title{A Constructive Proof of Masser's Theorem}

\author{Alexander J. Barrios}
\address{Department of Mathematics and Statistics, Carleton College, Northfield, Minnesota 55057}
\email{abarrios@carleton.edu}


\date{June 12, 2019}



\begin{abstract}
The Modified Szpiro Conjecture, equivalent to the $abc$ Conjecture, states
that for each $\epsilon>0$, there are finitely many rational elliptic curves
satisfying $N_{E}^{6+\epsilon}<\max\!\left\{  \left\vert c_{4}^{3}\right\vert
,c_{6}^{2}\right\}  $ where $c_{4}$ and $c_{6}$ are the invariants associated
to a minimal model of $E$ and $N_{E}$ is the conductor of $E$. We say $E$ is a
good elliptic curve if $N_{E}^{6}<\max\!\left\{  \left\vert c_{4}^{3}\right\vert
,c_{6}^{2}\right\}  $. Masser showed that there are infinitely many good Frey
curves. Here we give a constructive proof of this assertion.
\end{abstract}

\maketitle

\section{Introduction}

By an $ABC$ triple, we mean a triple of positive integers $\left(
a,b,c\right)  $ such that $a,b,$ and $c$ are relatively prime positive
integers with $a+b=c$. The $ABC$ Conjecture \cite[5.1]{MR1860012} states that
for any $\epsilon>0$, there are only finitely many $ABC$ triples that satisfy
$\operatorname{rad}\!\left(  abc\right)  ^{1+\epsilon}<c$ where
$\operatorname{rad}\!\left(  n\right)  $ denotes the product of the distinct
primes dividing $n$. We say that an $ABC$ triple is good if
$\operatorname{rad}\!\left(  abc\right)  <c$. For instance, the triple
$\left(  1,8,9\right)  $ is a good $ABC$ triple and more generally the triple
$\left(  1,9^{k}-1,9^{k}\right)  $ is a good $ABC$ triple for each positive
integer $k$ \cite{MR1860012}. In $1988,$ Oesterl\'{e} \cite{MR992208} proved
that the $ABC$ Conjecture is equivalent to the modified Szpiro conjecture
which states that for $\epsilon>0$, there are only finitely many elliptic
curves $E$ such that $N_{E}^{6+\epsilon}<\max\left\{  \left\vert c_{4}%
^{3}\right\vert ,c_{6}^{2}\right\}  $ where $N_{E}$ denotes the conductor of
the elliptic curve and $c_{4}$ and $c_{6}$ are the invariants associated to a
minimal model of $E$. As with $ABC$ triples, we define a good elliptic curve
to be an elliptic curve $E$ that satisfies the inequality $N_{E}^{6}%
<\max\left\{  \left\vert c_{4}^{3}\right\vert ,c_{6}^{2}\right\}  $. In the
special case of Frey curves, that is, a rational elliptic curve that has a
Weierstrass model of the form $y^{2}=x\left(  x-a\right)  \left(  x+b\right)
$ where $a$ and $b$ are relatively prime integers, Masser \cite{MR1065152}
showed that there are infinitely many good Frey curves. In this article, we
provide a constructive proof of Masser's Theorem. Moreover, the torsion
subgroup of a Frey curve can only take on four possibilities due to Mazur's
Torsion Theorem \cite{MR488287}, namely $E\!\left(
\mathbb{Q}
\right)  _{\text{tors}}\cong C_{2}\times C_{2N}$ where $C_{m}$ denotes the
cyclic group of order $m$ and $N=1,2,3,$ or $4$. With this we state our main theorem:

\begin{claim}
[Theorem 1]For each of the four possible torsion subgroups $T=C_{2}\times
C_{2N}$ where $N=1,2,3,$ or $4$, there are infinitely many good elliptic
curves such that $E\!\left(
\mathbb{Q}
\right)  _{\text{tors}}\cong T$.
\end{claim}

This is equivalent to Theorem \ref{mainthmmsr}, where the main theorem is
given in its constructive form. As a consequence we get examples akin to the
infinitely many good $ABC$ triples $\left(  1,9^{k}-1,9^{k}\right)  $ for each
positive integer $k$. For each of the four possible $T$, we use rational maps
of modular curves to construct a recursive sequence of $ABC$ triples
$P_{j}^{T}=\left(  a_{j},b_{j},c_{j}\right)  $ such that if $P_{j}^{T}$ is a
good $ABC$ triple satisfying certain congruences, then $P_{j}^{T}$ is a good
$ABC$ triple for each nonnegative integer $j$. Once this is proven, we prove
our main Theorem by showing that the associated Frey curve%
\[
F_{P_{j}^{T}}:y^{2}=x\left(  x-a_{j}\right)  \left(  x+b_{j}\right)
\]
is a good elliptic curve for each positive integer $j$ with $F_{P_{j}^{T}%
}\!\left(
\mathbb{Q}
\right)  _{\text{tors}}\cong T$.

\section{Certain Polynomials}

In this section we establish a series of technical results which will ease the
proofs in the sections that are to follow. Let $T=C_{2}\times C_{2N}$ where
$N=1,2,3,4$. For each $T$ let $\mathfrak{A}_{T}=\mathfrak{A}_{T}\!\left(
a,b\right)  ,$\ $\mathfrak{B}_{T}=\mathfrak{B}_{T}\!\left(  a,b\right)
,$\ $\mathfrak{C}_{T}=\mathfrak{C}_{T}\!\left(  a,b\right)  ,$\ $\mathfrak{D}%
_{T}=\mathfrak{D}_{T}\!\left(  a,b\right)  ,$\ $\mathfrak{A}_{T}%
^{r}=\mathfrak{A}_{T}^{r}\!\left(  a,b\right)  ,\ \mathfrak{B}_{T}%
^{r}=\mathfrak{B}_{T}^{r}\!\left(  a,b\right)  ,$\ $\mathfrak{C}_{T}%
^{r}=\mathfrak{C}_{T}^{r}\!\left(  a,b\right)  ,$\ $U_{T}=U_{T}\!\left(
a,b,r,s\right)  ,$\ $V_{T}=V_{T}\!\left(  a,b,r,s\right)  ,$\ and $W_{T}%
=W_{T}\left(  a,b,r,s\right)  $ be the polynomials in $R=%
\mathbb{Z}
\!\left[  a,b,r,s\right]  $ defined in Table \ref{ta:ABCtriplesPaper}.

For a fixed $T$, the polynomials $\mathfrak{A}_{T},$\ $\mathfrak{B}_{T}%
,$\ $\mathfrak{C}_{T},$ and\ $\mathfrak{D}_{T}$ are homogenous polynomials in
$a$ and $b$ of the same degree $m_{T}$. In particular, we have the equalities%
\[%
\begin{tabular}
[c]{lll}%
$a^{m_{T}}\mathfrak{A}_{T}\!\left(  1,\frac{b}{a}\right)  =\mathfrak{A}%
_{T}\!\left(  a,b\right)  $ & $\qquad$ & $a^{m_{T}}\mathfrak{B}_{T}\!\left(
1,\frac{b}{a}\right)  =\mathfrak{B}_{T}\!\left(  a,b\right)  $\\
$a^{m_{T}}\mathfrak{C}_{T}\!\left(  1,\frac{b}{a}\right)  =\mathfrak{C}%
_{T}\!\left(  a,b\right)  $ & $\qquad$ & $a^{m_{T}}\mathfrak{D}_{T}\!\left(
1,\frac{b}{a}\right)  =\mathfrak{D}_{T}\!\left(  a,b\right)  .$%
\end{tabular}
\
\]
The first result can be verified via a computer algebra system and we note
that we are considering $\mathfrak{A}_{T}\!\left(  1,t\right)  ,\ \mathfrak{B}%
_{T}\!\left(  1,t\right)  ,\ \mathfrak{C}_{T}\!\left(  1,t\right)
,\ \mathfrak{D}_{T}\!\left(  1,t\right)  $ as functions from $%
\mathbb{R}
$ to $%
\mathbb{R}
$.

\begin{lemma}
\label{polyn}\label{ap:ABCroots}\label{realvalued}For $T=C_{2}\times C_{2N}$
with $N=1,2,3,4$, let $f_{T},g_{T}:%
\mathbb{R}
\rightarrow%
\mathbb{R}
$ be the function in the variable $t$ defined in Table
\ref{ta:ABCtriplesPaper}. Let $\theta_{T}$ be the greatest real root of
$f_{T}\!\left(  t\right)  $. The (approximate) value of $\theta_{T}$ is found
in Table \ref{ta:ABCtriplesPaper}. Then for each $T$,

\begin{enumerate}
\item $\mathfrak{A}_{T}\in4R$;

\item $\mathfrak{A}_{T}+\mathfrak{B}_{T}=\mathfrak{C}_{T}$;

\item $U_{T}\mathfrak{B}_{T}+V_{T}\mathfrak{C}_{T}=W_{T}$;

\item $f_{T}\!\left(  \frac{b}{a}\right)  =\frac{\mathfrak{B}_{T}\!\left(
a,b\right)  }{\mathfrak{A}_{T}\left(  a,b\right)  }-\frac{b}{a}$;

\item $g_{T}\!\left(  t\right)  =\mathfrak{C}_{T}\!\left(  1,t\right)
-\mathfrak{D}_{T}\!\left(  1,t\right)  $;

\item $f_{T}\!\left(  t\right)  ,g_{T}\!\left(  t\right)  ,\ \mathfrak{A}%
_{T}\!\left(  1,t\right)  ,\ \mathfrak{B}_{T}\!\left(  1,t\right)
,\ \mathfrak{C}_{T}\!\left(  1,t\right)  ,\ \mathfrak{D}_{T}\!\left(
1,t\right)  >0$ for $t>\theta_{T}$;

\item For $T=C_{2}\times C_{2N}$ for $N=1,2$, $f_{T}\!\left(  t\right)
,g_{T}\!\left(  t\right)  ,\ \mathfrak{A}_{T}\!\left(  1,t\right)
,\ \mathfrak{B}_{T}\!\left(  1,t\right)  ,\ \mathfrak{C}_{T}\!\left(
1,t\right)  $, $\mathfrak{D}_{T}\!\left(  1,t\right)  >0$ for $t$ in $\left(
0,1\right)  $.
\end{enumerate}
\end{lemma}

\section{Good \bm{$ABC$} Triples}

\begin{definition}
By an $ABC$ triple, we mean a triple $P=\left(  a,b,c\right)  $ such that
$a,b,$ and $c$ are relatively prime positive integers with $a+b=c$. We say
$P=\left(  a,b,c\right)  $ is good if $\operatorname{rad}\!\left(  abc\right)
<c$.
\end{definition}

\begin{lemma}
\label{triples}For each $T=C_{2}\times C_{2N}$, let $P=\left(  a,b,a+b\right)
$ be an $ABC$ triple with $a$ even and $\frac{b}{a}>\theta_{T}$ where
$\theta_{T}$ is as defined in Lemma \ref{realvalued}. Suppose further that
$a\equiv0\ \operatorname{mod}3$ if $N=3$. Then $\left(  \mathfrak{A}%
_{T},\mathfrak{B}_{T},\mathfrak{C}_{T}\right)  $ is an $ABC$ triple with
$\mathfrak{A}_{T}\equiv0\ \operatorname{mod}16,\ \mathfrak{B}_{T}%
\equiv1\ \operatorname{mod}4,$ and$\ \frac{\mathfrak{B}_{T}}{\mathfrak{A}_{T}%
}>\theta_{T}$. Moreover, if $N=3$, then $\mathfrak{A}_{T}\equiv
0\ \operatorname{mod}3$.
\end{lemma}

\begin{proof}
Since $a$ and $b$ are relatively prime, there exist integers $r$ and $s$ such
that $ra^{n}+sb^{n}=1,$ for any positive integer $n$. Therefore, by Lemma
\ref{polyn}, $\gcd\!\left(  \mathfrak{B}_{T},\mathfrak{C}_{T}\right)  $
divides $32$ if $N\neq3$ and $\gcd\!\left(  \mathfrak{B}_{T},\mathfrak{C}%
_{T}\right)  $ divides $48$ if $N=3$. Since $a$ is even and $a\equiv
0\ \operatorname{mod}3$ when $N=3$, we conclude that $\gcd\!\left(
\mathfrak{B}_{T},\mathfrak{C}_{T}\right)  =1$. Next, observe that%
\[
f_{T}\!\left(  \frac{b}{a}\right)  =\frac{\mathfrak{B}_{T}\!\left(  1,\frac
{b}{a}\right)  }{\mathfrak{A}_{T}\!\left(  1,\frac{b}{a}\right)  }-\frac{b}%
{a}.
\]
Since $\frac{b}{a}>\theta_{T}$, we have by Lemma \ref{realvalued} that
$f_{T}\!\left(  \frac{b}{a}\right)  $ is positive and therefore $\frac
{\mathfrak{B}_{T}}{\mathfrak{A}_{T}}>\frac{b}{a}>\theta_{T}$. By Lemma
\ref{polyn} we also have that $\mathfrak{A}_{T}+\mathfrak{B}_{T}%
=\mathfrak{C}_{T}$ for each $T$ and therefore $\left(  \mathfrak{A}%
_{T},\mathfrak{B}_{T},\mathfrak{C}_{T}\right)  $ is an $ABC$ triple. Since $a$
is even it is easily verified that $\mathfrak{A}_{T}\equiv
0\ \operatorname{mod}16$. Similarly, when $N=3$, $\mathfrak{A}_{T}%
\equiv0\ \operatorname{mod}3$ since $a\equiv0\ \operatorname{mod}3$. It easily
checked that for each $T$, $\mathfrak{B}_{T}\equiv b^{2k}\ \operatorname{mod}%
4$ for some integer $k$. Since $b$ is odd, it follows that $\mathfrak{B}%
_{T}\equiv1\ \operatorname{mod}4$.
\end{proof}

\begin{lemma}
\label{ap:lemmarad}Let $P=\left(  a,b,a+b\right)  $ be a good $ABC$ triple and
assume the statement of Lemma \ref{triples}. Then $\left(  \mathfrak{A}%
_{T},\mathfrak{B}_{T},\mathfrak{C}_{T}\right)  $ is a good $ABC$ triple.
\end{lemma}

\begin{proof}
Since $a$ is assumed to be even, we have that $\operatorname{rad}\!\left(
2^{n}ax\right)  =\operatorname{rad}\!\left(  ax\right)  $ for some integer
$x$. Therefore%
\[
\operatorname{rad}\!\left(  \mathfrak{A}_{T}\right)  =\operatorname{rad}%
\!\left(  \mathfrak{A}_{T}^{r}\right)  ,\qquad\operatorname{rad}\!\left(
\mathfrak{B}_{T}\right)  =\operatorname{rad}\!\left(  \mathfrak{B}_{T}%
^{r}\right)  ,\qquad\operatorname{rad}\!\left(  \mathfrak{C}_{T}\right)
=\operatorname{rad}\!\left(  \mathfrak{C}_{T}^{r}\right)  .
\]
Since $\left(  a,b,a+b\right)  $ is a good $ABC$ triple, we have that
$\operatorname{rad}\!\left(  ab\!\left(  a+b\right)  \right)  <a+b$. From this
and the fact that $\operatorname{rad}\!\left(  xy^{k}\right)
=\operatorname{rad}\!\left(  xy\right)  \leq xy$ for positive integers
$k,x,y$, we have that for each $T$, we attain%
\[
\operatorname{rad}\!\left(  \mathfrak{A}_{T}\mathfrak{B}_{T}\mathfrak{C}%
_{T}\right)  =\operatorname{rad}\!\left(  \mathfrak{A}_{T}^{r}\mathfrak{B}%
_{T}^{r}\mathfrak{C}_{T}^{r}\right)  <\left\vert \mathfrak{D}_{T}\right\vert
.
\]
Since $\frac{b}{a}>\theta_{T}$, $\mathfrak{D}_{T}\!\left(  1,\frac{b}%
{a}\right)  $ is positive by Lemma \ref{polyn}. In particular, $\mathfrak{D}%
_{T}$ is positive since $a^{m_{T}}\mathfrak{D}_{T}\!\left(  1,\frac{b}%
{a}\right)  =\mathfrak{D}_{T}$ where $m_{T}$ is the homogenous degree of
$\mathfrak{D}_{T}$. Now observe that%
\[
\mathfrak{C}_{T}-\operatorname{rad}\!\left(  \mathfrak{A}_{T}\mathfrak{B}%
_{T}\mathfrak{C}_{T}\right)  >\mathfrak{C}_{T}-\mathfrak{D}_{T}=a^{m_{T}%
}\left(  \mathfrak{C}_{T}\!\left(  1,\frac{b}{a}\right)  -\mathfrak{D}%
_{T}\!\left(  1,\frac{b}{a}\right)  \right)  >0
\]
where the positivity follows from Lemma \ref{polyn}. Hence $\left(
\mathfrak{A}_{T},\mathfrak{B}_{T},\mathfrak{C}_{T}\right)  $ is a good $ABC$
triple since $\operatorname{rad}\!\left(  \mathfrak{A}_{T}\mathfrak{B}%
_{T}\mathfrak{C}_{T}\right)  <\mathfrak{C}_{T}$.
\end{proof}

\begin{proposition}
\label{ap:ABCrec}Let $\left(  a_{0},b_{0},c_{0}\right)  $ be a good $ABC$
triple with $a_{0}$ even. For each $T$ define the triple $P_{j}^{T}$
recursively by%
\[
P_{j}^{T}=\left(  a_{j},b_{j},c_{j}\right)  =\left(  \mathfrak{A}_{T}\!\left(
a_{j-1},b_{j-1}\right)  ,\mathfrak{B}_{T}\!\left(  a_{j-1},b_{j-1}\right)
,\mathfrak{C}_{T}\left(  a_{j-1},b_{j-1}\right)  \right)  \qquad\text{for
}j\geq1.
\]
Assume further that $\frac{b_{0}}{a_{0}}>\theta_{T}$ and that $b_{0}%
\equiv0\ \operatorname{mod}3$ if $T=C_{2}\times C_{6}$. Then for each $j\geq
1$, $P_{j}^{T}$ is a good $ABC$ triple with $a_{j}\equiv0\ \operatorname{mod}%
16,\ b_{j}\equiv1\ \operatorname{mod}4,$ and $\frac{b_{j}}{a_{j}}>\theta_{T}$.
Additionally, if $T=C_{2}\times C_{6}$, then $a_{j}\equiv0\ \operatorname{mod}%
3$.
\end{proposition}

\begin{proof}
This follows automatically from Lemmas \ref{triples} and \ref{ap:lemmarad}.
\end{proof}

\section{Frey Curves\label{FreyCurveABCpaper}}

As before, we suppose $T=C_{2}\times C_{2N}$ and define for $t\in
\mathbb{P}^{1}$, the mapping $\mathcal{X}_{t}$ as the mapping which takes $T$
to the elliptic curve $\mathcal{X}_{t}\!\left(  T\right)  $ where the
Weierstrass model of $\mathcal{X}_{t}\!\left(  T\right)  $ is given in Table
\ref{ta:universalell}. Our parameterizations for $T=C_{2}\times C_{2N}$ where
$N=3,4$ are those found in \cite[Table 3]{MR1748483} which expands the
implicit expressions for the parameters $b$ and $c$ in \cite[Table
3]{MR0434947} to express the universal elliptic curves for the modular curves
$X_{1}\!\left(  2,2N\right)  $ in terms of a single parameter $t$. Similarly,
our model for $T=C_{2}\times C_{4}$ differs by a linear change of variables
from the model given for $W_{4}$ in \cite[$\S 4$]{MR1638488} which
parameterizes elliptic curves $E$ with $C_{4}\times C_{4}\hookrightarrow
E\!\left(
\mathbb{Q}
\!\left(  i\right)  \right)  _{\text{tors}}$. In particular, $\mathcal{X}%
_{t}\!\left(  T\right)  $ is a one-parameter family of elliptic curves with
the property that if $t\in K$ for some field $K$, then $\mathcal{X}%
_{t}\!\left(  T\right)  $ is an elliptic curve over $K$ and $T\hookrightarrow
\mathcal{X}_{t}\!\left(  T\right)  \!\left(  K\right)  _{\text{tors}}$.

\begin{table}[h]
\caption{Universal Elliptic Curve $\mathcal{X}_{t}\!\left(  T\right)  $}%
\label{ta:universalell}
\begin{center}%
\begin{tabular}
[c]{|c|c|c|}\hline
\multicolumn{3}{|c|}{$\mathcal{X}_{t}\!\left(  T\right)  :y^{2}+\left(
1-g\right)  xy-fy=x^{3}-fx^{2}$}\\\hline
$f$ & $g$ & $T$\\\hline\hline
$\frac{2t^{4}-7t^{3}+12t^{2}-7t+2}{2(-1+t)^{4}}$ & $1$ & $C_{2}\times C_{4}%
$\\\hline
$\frac{-2t^{3}+14t^{2}-22t+10}{\left(  t+3\right)  ^{2}\left(  t-3\right)
^{2}}$ & $\frac{-2t+10}{\left(  t+3\right)  \left(  t-3\right)  }$ &
$C_{2}\times C_{6}$\\\hline
$\frac{16t^{3}+16t^{2}+6t+1}{\left(  8t^{2}-1\right)  ^{2}}$ & $\frac
{16t^{3}+16t^{2}+6t+1}{2t\left(  4t+1\right)  \left(  8t^{2}-1\right)  }$ &
$C_{2}\times C_{8}$\\\hline
\end{tabular}
\end{center}
\end{table}For $T=C_{2}\times C_{2}$, define%
\begin{equation}
\mathcal{X}_{t}\!\left(  T\right)  :y^{2}=x^{3}+\left(  t^{4}-12t^{3}%
+6t^{2}-12t+1\right)  x^{2}-8t\left(  t-1\right)  ^{4}\left(  t^{2}+1\right)
x. \label{C2xC2}%
\end{equation}

\begin{lemma}
\label{LemmaunivincT}If $t\in%
\mathbb{Q}
$ such that $\mathcal{X}_{t}\!\left(  T\right)  $ is an elliptic curve, then
$T\hookrightarrow\mathcal{X}_{t}\!\left(  T\right)  \!\left(
\mathbb{Q}
\right)  _{\text{tors}}$.
\end{lemma}

\begin{proof}
Recall that the modular curve $X_{1}\!\left(  2,2N\right)  $ (with cusps
removed) for $N=2,3,4$ parameterizes isomorphism classes of pairs $\left(
E,P,Q\right)  $ where $E$ is an elliptic curve having full $2$-torsion, $P$
and $Q$ are torsion points of order $2$ and $2N$, respectively, and
$\left\langle P,N\cdot Q\right\rangle =E\!\left[  2\right]  $.

For $T=C_{2}\times C_{2N}$ where $N=3,4$, we note that our parameterizations
are those of the universal elliptic curve for the modular curve $X_{1}%
\!\left(  2,2N\right)  $ \cite[Table 3]{MR1748483}. Thus $T\hookrightarrow
\mathcal{X}_{t}\!\left(  T\right)  \left(
\mathbb{Q}
\right)  _{\text{tors}}$.

For $T=C_{2}\times C_{4}$, let%
\[
t^{\prime}=\frac{t}{2\left(  t-1\right)  ^{2}}%
\]
so that $\mathcal{X}_{t}\!\left(  T\right)  $ is equal to the Weierstrass
model given for the universal elliptic curve over $X_{1}\!\left(  2,4\right)
$ given in \cite[Table 3]{MR1748483} with parameter $t^{\prime}$. Hence
$T\hookrightarrow\mathcal{X}_{t}\!\left(  T\right)  \!\left(
\mathbb{Q}
\right)  _{\text{tors}}$.

For $T=C_{2}\times C_{2}$, let $t=\frac{b}{a}$ and consider the admissible
change of variables $x\longmapsto\frac{1}{a^{4}}x$ and $y\longmapsto\frac
{1}{a^{6}}y$. This gives a $%
\mathbb{Q}
$-isomorphism between $\mathcal{X}_{t}\!\left(  T\right)  $ and the elliptic
curve%
\[
y^{2}=x\left(  x-8ab\left(  a^{2}+b^{2}\right)  \right)  \left(  x+\left(
a-b\right)  ^{4}\right)
\]
which has $\left\langle \left(  8ab\left(  a^{2}+b^{2}\right)  ,0\right)
,\left(  0,0\right)  \right\rangle \cong C_{2}\times C_{2}$. Thus
$T\hookrightarrow\mathcal{X}_{t}\!\left(  T\right)  \!\left(
\mathbb{Q}
\right)  _{\text{tors}}$.
\end{proof}

\begin{definition}
For an $ABC$ triple $P=\left(  a,b,c\right)  $, let $F_{P}=F_{P}\!\left(
a,b\right)  $ be the Frey curve given by the Weierstrass model%
\[
F_{P}:y^{2}=x\left(  x-a\right)  \left(  x+b\right)  .
\]

\end{definition}

\begin{lemma}
\label{ap:lemmaFreytors}Let $\left(  a,b,c\right)  $ be an $ABC$ triple which
satisfies the assumptions of Lemma \ref{triples}. Then for each $T$, the Frey
curve $F_{P}$ with $P=\left(  \mathfrak{A}_{T},\mathfrak{B}_{T},\mathfrak{C}%
_{T}\right)  $ has torsion subgroup $F_{P}\!\left(
\mathbb{Q}
\right)  _{\text{tors}}\cong T$.
\end{lemma}

\begin{proof}
Let $\mathcal{X}_{t}\!\left(  T\right)  $ be as defined in Table
\ref{ta:universalell} for $T=C_{2}\times C_{2N}$ for $N=2,3,4$ and as defined
in (\ref{C2xC2}) for $N=1$. In addition, let $u_{T},r_{T},s_{T},w_{T},$ and
$t_{T}$ be as defined in Table \ref{ta:bestabclemfreyad}. We now proceed by cases.

Case I. Suppose $T=C_{2}\times C_{2N}$ for $N=2,3,4$. Then the admissible
change of variables $x\longmapsto u_{T}^{2}x+r_{T}$ and $y\longmapsto
u_{T}^{3}y+u_{T}^{2}s_{T}x+w_{T}$ gives a $%
\mathbb{Q}
$-isomorphism from $F_{P}$ onto $\mathcal{X}_{t_{T}}\!\left(  T\right)  $. In
particular, $T\hookrightarrow F_{P}\!\left(
\mathbb{Q}
\right)  _{\text{tors}}$ by Lemma \ref{LemmaunivincT}. By Mazur's Torsion
Theorem \cite{MR488287} we conclude that $F_{P}\!\left(
\mathbb{Q}
\right)  _{\text{tors}}\cong C_{2}\times C_{2N}$ for $N=3,4$ and that
$F_{P}\!\left(
\mathbb{Q}
\right)  _{\text{tors}}$ is isomorphic to either $C_{2}\times C_{4}$ or
$C_{2}\times C_{8}$ if $T=C_{2}\times C_{4}$. For the latter, we observe that
our model for $\mathcal{X}_{t}\!\left(  T\right)  $ parametrizes elliptic
curves $E$ over $%
\mathbb{Q}
\!\left(  i\right)  $ with $C_{4}\times C_{4}\hookrightarrow E\!\left(
\mathbb{Q}
\!\left(  i\right)  \right)  _{\text{tors}}$ \cite[$\S 4$]{MR1638488}. By
Kamienny's Torsion Theorem \cite{MR1172689} we conclude that $E\!\left(
\mathbb{Q}
\!\left(  i\right)  \right)  _{\text{tors}}\cong C_{4}\times C_{4}$. Thus
$\mathcal{X}_{t}\!\left(  T\right)  \!\left(
\mathbb{Q}
\!\left(  i\right)  \right)  _{\text{tors}}\cong C_{4}\times C_{4}$ and
therefore $C_{2}\times C_{8}\not \hookrightarrow \mathcal{X}_{t}\!\left(
T\right)  \!\left(
\mathbb{Q}
\!\left(  i\right)  \right)  _{\text{tors}}$. Hence $\mathcal{X}_{t}\!\left(
T\right)  \!\left(
\mathbb{Q}
\right)  _{\text{tors}}\cong C_{2}\times C_{4}$.

Case II. Suppose $T=C_{2}\times C_{2}$ and $T_{4}=C_{2}\times C_{4}$. Then
there is a $2$-isogeny $\phi:\mathcal{X}_{t}\!\left(  T_{4}\right)
\rightarrow\mathcal{X}_{t}\!\left(  T\right)  $ obtained by applying
V\'{e}lu's formulas \cite{MR0294345} to the elliptic curve $\mathcal{X}%
_{t}\!\left(  T_{4}\right)  $ and its torsion point $2P$ where $P=\left(
0,0\right)  $ is the torsion point of order $4$ of $\mathcal{X}_{t}\!\left(
T_{4}\right)  $.

Next, observe that via the First Isomorphism Theorem:%
\begin{equation}
\left\vert \mathcal{X}_{t}\!\left(  T\right)  \!\left(
\mathbb{Q}
\right)  _{\text{tors}}\right\vert \left\vert \mathcal{X}_{t}\!\left(
T_{4}\right)  \!\left(
\mathbb{Q}
\right)  \!\left[  \phi\right]  \right\vert =\left\vert \mathcal{X}%
_{t}\!\left(  T_{4}\right)  \!\left(
\mathbb{Q}
\right)  _{\text{tors}}\right\vert \left[  \mathcal{X}_{t}\!\left(
T_{4}\right)  \!\left(
\mathbb{Q}
\right)  _{\text{tors}}:\phi\!\left(  \mathcal{X}_{t}\!\left(  T\right)
\!\left(
\mathbb{Q}
\right)  _{\text{tors}}\right)  \right]  . \label{1stisoiso}%
\end{equation}

By Case I above we have that $\left\vert \mathcal{X}_{t}\!\left(
T_{4}\right)  \!\left(
\mathbb{Q}
\right)  _{\text{tors}}\right\vert =8$ which implies that the only prime
dividing $\left\vert \mathcal{X}_{t}\!\left(  T\right)  \!\left(
\mathbb{Q}
\right)  _{\text{tors}}\right\vert $ is $2$ since $\phi$ is a $2$-isogeny.

Next, we consider the admissible change of variables $x\longmapsto u_{T}%
^{2}x+r_{T}$ and $y\longmapsto u_{T}^{3}y+u_{T}^{2}s_{T}x+w_{T}$ which gives a
$%
\mathbb{Q}
$-isomorphism from $F_{P}$ onto $\mathcal{X}_{t_{T}}\!\left(  T\right)  $. In
particular, $C_{2}\times C_{2}\hookrightarrow F_{P}\!\left(
\mathbb{Q}
\right)  _{\text{tors}}$ by Lemma \ref{LemmaunivincT}. By the proof of Lemma
\ref{LemmaunivincT}, $\mathcal{X}_{t}\!\left(  T\right)  $ is $%
\mathbb{Q}
$-isomorphic to the elliptic curve given by the Weierstrass model%
\[
y^{2}=x\left(  x-8ab\left(  a^{2}+b^{2}\right)  \right)  \left(  x+\left(
a-b\right)  ^{4}\right)  .
\]
This model satisfies the assumptions of \cite[Main Theorem 1]{MR1424534} and
therefore we have that $\mathcal{X}_{t}\!\left(  T\right)  \!\left(
\mathbb{Q}
\right)  _{\text{tors}}\cong C_{2}\times C_{2}$ if $8ab\left(  a^{2}%
+b^{2}\right)  $ is not a square. If it were a square we would have a
nontrivial integer solution to the Diophantine equation $x^{4}-y^{4}=z^{2}$
since%
\[
8ab\left(  a^{2}+b^{2}\right)  +\left(  a-b\right)  ^{4}=\left(  a+b\right)
^{4}.
\]
This contradicts Fermat's Theorem and therefore $\mathcal{Y}_{t}\!\left(
T\right)  \!\left(
\mathbb{Q}
\right)  _{\text{tors}}\cong C_{2}\times C_{2}$.
\end{proof}

\begin{theorem}
\label{ABCmainthm}Let $T=C_{2}\times C_{2N}$ for $N=1,2,3,4$ and consider the
sequence of good $ABC$ triples $P_{j}^{T}$ defined in Proposition
\ref{ap:ABCrec}. Then for each $j\geq1$, the Frey curve $F_{P_{j}^{T}}$
determined by $P_{j}^{T}$ has torsion subgroup $F_{P_{j}^{T}}\!\left(
\mathbb{Q}
\right)  _{\text{tors}}\cong C_{2}\times C_{2N}$.
\end{theorem}

\begin{proof}
In Proposition \ref{ap:ABCrec}, we saw that each $P_{j}^{T}$ satisfies the
assumptions of Lemma \ref{triples}. Consequently, the Theorem follows from
Lemma \ref{ap:lemmaFreytors}.
\end{proof}

The case of $N=2,4$ in Theorem \ref{ABCmainthm} was proven by the author
alongside Watts and Tillman \cite{msriup} as part of the Mathematical Sciences
Research Institute Undergraduate Program.

\section{Examples of Good ABC Triples}

\begin{definition}
For an $ABC$\textbf{ }triple $P=\left(  a,b,c\right)  $, define the quality
$q\!\left(  P\right)  $ of $P$ to be
\[
q\!\left(  P\right)  =\frac{\log\!\left(  c\right)  }{\log\!\left(
\operatorname{rad}\!\left(  abc\right)  \right)  }.
\]
In particular, $P$ is a good $ABC$ triple is equivalent to $q\!\left(
P\right)  >1$.
\end{definition}

\begin{example}
\label{ap:exofcurves}For $T=C_{2}\times C_{2N}$ where $N=1,2$ let
$P_{0}=\left(  2^{5},7^{2},3^{4}\right)  $. Then $P_{0}$ is a good $ABC$
triple since $q\!\left(  P\right)  \approx1.1757$. By Proposition
\ref{ap:ABCrec}, this good $ABC$ triple results in two distinct infinite
sequences of good $ABC$ triples $P_{j}^{T}$.

For $T=C_{2}\times C_{6}$, let $P_{0}=\left(  2^{4}3^{3},17^{3}61,5^{3}%
7^{4}\right)  $. Then $P_{0}$ is a good $ABC$ triple since $q\!\left(
P\right)  \approx1.0261$. Moreover, $\frac{17^{3}61}{2^{4}3^{3}}>\theta_{T}$.
By Proposition \ref{ap:ABCrec}, this good $ABC$ triple results in an infinite
sequence of good $ABC$ triples $P_{j}^{T}$.

For $T=C_{2}\times C_{8}$, let $P_{0}=\left(  2^{2},11^{2},5^{3}\right)  $.
Then $P_{0}$ is a good $ABC$ triple since $q\!\left(  P\right)  \approx
1.0272$. Moreover, $\frac{121}{4}>\theta_{T}$. By Proposition \ref{ap:ABCrec},
this good $ABC$ triple results in an infinite sequence of good $ABC$ triples
$P_{j}^{T}$.

Table \ref{ta:apbestexa} gives $a_{1}$ and $b_{1}$ of $P_{j}^{T}=\left(
a_{j},b_{j},c_{j}\right)  $ as well as the quality $q\!\left(  P_{j}%
^{T}\right)  $ for $j=1,2,3$. We note that the values of $a_{j}$ and $b_{j}$
are not given for $j\geq2$ due to the size of these quantities. For
$T=C_{2}\times C_{2N}$ for $N=3,4$, we only compute $q\!\left(  P_{j}%
^{T}\right)  $ for $j=1,2$ due to computational limitations. \begin{table}[h]
\caption{Table for Example \ref{ap:exofcurves}}%
\label{ta:apbestexa}
\begin{center}%
\begin{tabular}
[c]{|c||c|c|c|c|c|}\hline
$T$ & $a_{1}$ & $b_{1}$ & $q\!\left(  P_{1}^{T}\right)  $ & $q\!\left(
P_{2}^{T}\right)  $ & $q\!\left(  P_{3}^{T}\right)  $\\\hline\hline
$C_{2}\times C_{2}$ & $2^{5}11^{2}14657$ & $3^{8}13^{4}$ & $1.0755$ & $1.0324$
& $1.015$\\\hline
$C_{2}\times C_{4}$ & $2^{12}7^{4}$ & $3^{8}17^{2}$ & $1.2425$ & $1.0531$ &
$1.0130$\\\hline
$C_{2}\times C_{6}$ & $2^{16}3^{9}17^{3}61$ & $5^{9}7^{12}11\cdot27127$ &
$1.1211$ & $1.0278$ & $-$\\\hline
$C_{2}\times C_{8}$ & $2^{12}11^{8}$ & $7\cdot31\cdot503\cdot1951\cdot
14657^{2}$ & $1.0331$ & $1.0040$ & $-$\\\hline
\end{tabular}
\end{center}
\end{table}
\end{example}

\section{Infinitely Many Good Frey Curves\label{ch:infgoodtor}}

Recall that the $ABC$ Conjecture is equivalent to the modified Szpiro
conjecture which states that for every $\epsilon>0$ there are finitely many
rational elliptic curves $E$ satisfying%
\[
N_{E}^{6+\epsilon}<\max\!\left\{  \left\vert c_{4}^{3}\right\vert ,c_{6}%
^{2}\right\}
\]
where $N_{E}$ is the conductor of $E$ and $c_{4}$ and $c_{6}$ are the
invariants associated to a minimal model of $E$. The following definition
gives the analog of good $ABC$ triples and the quality of an $ABC$ triple in
the context of elliptic curves.

\begin{definition}
Let $E$ be a rational elliptic curve with minimal discriminant $\Delta
_{E}^{\text{min}}$ and associated invariants $c_{4}$ and $c_{6}$. Define the
\textbf{modified Szpiro ratio} $\sigma_{m}\!\left(  E\right)  $ and
\textbf{Szpiro ratio} $\sigma\!\left(  E\right)  $ of $E$ to be the quantities%
\[
\sigma_{m}\!\left(  E\right)  =\frac{\log\max\left\{  \left\vert c_{4}%
^{3}\right\vert ,c_{6}^{2}\right\}  }{\log N_{E}}\qquad\text{and}\qquad
\sigma\!\left(  E\right)  =\frac{\log\left\vert \Delta_{E}^{\text{min}%
}\right\vert }{\log N_{E}}%
\]
where $N_{E}$ is the conductor of $E$. We say that $E$ is \textbf{good} if
$\sigma_{m}\!\left(  E\right)  >6$.
\end{definition}

Let $P=\left(  a,b,c\right)  $ be an $ABC$ triple with $a$ even and $b\equiv1$
$\operatorname{mod}4$. For $T=C_{2}\times C_{2N}$ where $N=1,2,3,4$, let
$\mathfrak{A}_{T}=\mathfrak{A}_{T}\!\left(  a,b\right)  ,$ $\mathfrak{B}%
_{T}=\mathfrak{B}_{T}\!\left(  a,b\right)  ,\ \mathfrak{C}_{T}=C_{T}\!\left(
a,b\right)  ,$ and $\mathfrak{D}_{T}=\mathfrak{D}_{T}\!\left(  a,b\right)  $
be as defined in Table \ref{ta:ABCtriplesPaper}. Assume further that
$a\equiv0\ \operatorname{mod}3$ if $T=C_{2}\times C_{6}$. Then the elliptic
curve $F_{T}=F_{T}\!\left(  a,b\right)  $ given by the Weierstrass model%
\[
F_{T}:y^{2}=x\left(  x-\mathfrak{A}_{T}\right)  \!\left(  x+\mathfrak{B}%
_{T}\right)
\]
satisfies $F_{T}\!\left(
\mathbb{Q}
\right)  _{\text{tors}}\cong T$ by Lemma \ref{ap:lemmaFreytors}. Moreover, the
congruences on $\mathfrak{A}_{T}$ and $\mathfrak{B}_{T}$ imply that the Frey
curve $F_{T}$ is semistable with minimal discriminant $\Delta_{T}=\left(
16^{-1}\mathfrak{A}_{T}\mathfrak{B}_{T}\mathfrak{C}_{T}\right)  ^{2}$
\cite[Exercise 8.23]{MR2514094}. Consequently, the conductor $N_{T}%
\mathfrak{\ }$of $F_{T}$ satisfies $N_{T}=\operatorname{rad}\!\left(
\Delta_{T}\right)  <\left\vert \mathfrak{D}_{T}\right\vert $ and the invariant
$c_{4,T}=c_{4,T}\!\left(  a,b\right)  $ associated with a global minimal model
of $F_{T}$ is as given in Table \ref{ta:c4THT}.

\begin{longtable}{|p{5in}|c|}
\caption{The Invariant $c_{4}$ of $F_{T}$}\\
\hline
$c_{4,T}$ & $T$\\
\hline
\endfirsthead
\caption[]{\emph{continued}}\\
\hline
$c_{4,T}$ & $T$\\
\hline
\endhead
\hline
\multicolumn{2}{r}{\emph{continued on next page}}
\endfoot
\hline
\endlastfoot
$a^{8}+60a^{6}b^{2}+134a^{4}b^{4}+60a^{2}b^{6}+b^{8}$ & $C_{2}\times C_{2}%
$\\\hline
$a^{8}+14a^{4}b^{4}+b^{8}$ & $C_{2}\times C_{4}$\\\hline
$9a^{8}+228a^{6}b^{2}+30a^{4}b^{4}-12a^{2}b^{6}+b^{8}$ & $C_{2}\times C_{6}%
$\\\hline
$a^{16}-8a^{14}b^{2}+12a^{12}b^{4}+8a^{10}b^{6}+230a^{8}b^{8}+8a^{6}%
b^{10}+12a^{4}b^{12}-8a^{2}b^{14}+b^{16}$ & $C_{2}\times C_{8}$
\label{ta:c4THT}	
\end{longtable}

\begin{lemma}
\label{goodFrey}Let $P=\left(  a,b,c\right)  $ be a good $ABC$ triple
satisfying $a\equiv0\ \operatorname{mod}2$, $b\equiv1\ \operatorname{mod}4,$
and $\frac{b}{a}>\theta_{T}$ where $\theta_{T}$ is as given in Lemma
\ref{ap:ABCroots}. Assume further that $a\equiv0\ \operatorname{mod}3$ if
$T=C_{2}\times C_{6}$. Then the Frey curve $F_{T}=F_{T}\!\left(
\mathfrak{A}_{T},\mathfrak{B}_{T}\right)  $ is good and $F_{T}\!\left(
\mathbb{Q}
\right)  _{\text{tors}}\cong T$.
\end{lemma}

\begin{proof}
By Lemma \ref{ap:lemmaFreytors}, $F_{T}\!\left(
\mathbb{Q}
\right)  _{\text{tors}}\cong T$. Since $F_{T}$ is a Frey curve we have that
the invariants $c_{4}$ and $c_{6}$ associated to a global minimal model of
$F_{T}$ satisfy $\max\!\left\{  \left\vert c_{4}^{3}\right\vert ,c_{6}%
^{2}\right\}  =c_{4}^{3}$ since $c_{4}$ and $\Delta_{F_{T}}^{\text{min}}$ are
always positive \cite[Lemma VIII.11.3]{MR2514094}. The congruences on $a$ and
$b$ imply that $c_{4}=c_{4,T}$. It, therefore, suffices to show that
$c_{4,T}^{3}-N_{T}^{6}>0$ where $N_{T}$ is the conductor of $F_{T}$. Since
$F_{T}$ is semistable,%
\[
N_{T}=\operatorname{rad}\!\left(  \mathfrak{A}_{T}\mathfrak{B}_{T}%
\mathfrak{C}_{T}\right)  <\mathfrak{D}_{T}%
\]
by Lemma \ref{ap:lemmarad}. Note that $\mathfrak{D}_{T}$ is positive since
$\frac{b}{a}>\theta_{T}$. Thus%
\begin{equation}
\frac{c_{4,T}^{3}-N_{T}^{6}}{\mathfrak{D}_{T}\!\left(  1,t\right)  ^{6}}%
>\frac{c_{4,T}\!\left(  1,t\right)  ^{3}-\mathfrak{D}_{T}\!\left(  1,t\right)
^{6}}{\mathfrak{D}_{T}\!\left(  1,t\right)  ^{6}}\text{ for }t=\frac{b}{a}
\label{chmsrlemabc}%
\end{equation}
Lastly, for each $T$, the polynomial $c_{4,T}\!\left(  1,t\right)
^{3}-\mathfrak{D}_{T}\!\left(  1,t\right)  ^{6}$ is positive on the open
interval $\left(  \theta_{T},\infty\right)  $ from which we conclude that
$F_{T}$ is a good elliptic curve.
\end{proof}

\begin{theorem}
\label{mainthmmsr}For each $T$, let $P_{0}^{T}=\left(  a_{0},b_{0}%
,c_{0}\right)  $ be a good $ABC$ triple satisfying $a_{0}\equiv
0\ \operatorname{mod}2$, $b_{0}\equiv1\ \operatorname{mod}4,$ and $\frac
{b_{0}}{a_{0}}>\theta_{T}$ where $\theta_{T}$ is as given in Lemma
\ref{ap:ABCroots}. Assume further that $a_{0}\equiv0\ \operatorname{mod}3$ if
$T=C_{2}\times C_{6}$. For $j\geq1$, define $P_{j}^{T}$ recursively by%
\[
P_{j}^{T}=\left(  a_{j},b_{j},c_{j}\right)  =\left(  \mathfrak{A}_{T}\!\left(
a_{j-1},b_{j-1}\right)  ,\mathfrak{B}_{T}\!\left(  a_{j-1},b_{j-1}\right)
,\mathfrak{C}_{T}\left(  a_{j-1},b_{j-1}\right)  \right)  .
\]
Then for each $j$, the Frey curve $F_{T}\!\left(  a_{j},b_{j}\right)  $ is
good and $F_{T}\!\left(  a_{j},b_{j}\right)  \left(
\mathbb{Q}
\right)  _{\text{tors}}\cong T$.
\end{theorem}

\begin{proof}
By Proposition \ref{ap:ABCrec}, $P_{j}^{T}=\left(  a_{j},b_{j},c_{j}\right)  $
satisfies $a_{j}\equiv0\ \operatorname{mod}2$, $b_{j}\equiv
1\ \operatorname{mod}4,$ and $\frac{b_{j}}{a_{j}}>\theta_{T}$ for each $j$.
For $T=C_{2}\times C_{6}$, if $a_{0}\equiv0\ \operatorname{mod}3$, then
$a_{j}\equiv0\ \operatorname{mod}3$ for each $j$. Hence $P_{j}^{T}$ is a good
$ABC$ triple for each $j$ by Proposition \ref{ap:ABCrec}. Therefore the result
follows by Lemma \ref{goodFrey}.
\end{proof}

In Example \ref{ap:exofcurves} we began with a good $ABC$ triple
$P_{0}=\left(  a_{0},b_{0},c_{0}\right)  $. For each $T$, we constructed an
infinite sequence of good $ABC$ triples $P_{j}^{T}=\left(  a_{j},b_{j}%
,c_{j}\right)  $. By Theorem \ref{mainthmmsr}, each Frey curve $F_{T}\!\left(
a_{j},b_{j}\right)  \left(
\mathbb{Q}
\right)  _{\text{tors}}$ is a good elliptic curve with torsion subgroup
isomorphic to $T$. Table \ref{ta:MSRandSRexA} lists the modified Szpiro ratios
of the Frey curves corresponding to $P_{j}^{T}$. Due to computational
limitations, we could only compute these ratios up to $j=3$.

\vspace{-1em}

\begin{table}[h]
\caption{Example of Good Frey Curves}%
\label{ta:MSRandSRexA}
\begin{center}%
\begin{tabular}
[c]{|c|c|c|cc}\hline
$T$ & $C_{2}\times C_{2}$ & $C_{2}\times C_{4}$ & $C_{2}\times C_{6}$ &
\multicolumn{1}{|c|}{$C_{2}\times C_{8}$}\\\hline
\multicolumn{1}{|r|}{$\sigma_{m}\!\left(  F_{T}\!\left(  a_{1},b_{1}\right)
\right)  $} & $6.4204$ & $7.4219$ & $6.7269$ & \multicolumn{1}{|c|}{$6.1985$%
}\\\hline
\multicolumn{1}{|r|}{$\sigma_{m}\!\left(  F_{T}\!\left(  a_{2},b_{2}\right)
\right)  $} & $6.1912$ & $6.3124$ & $6.1666$ & \multicolumn{1}{|c|}{$6.0241$%
}\\\hline
\multicolumn{1}{|r|}{$\sigma_{m}\!\left(  F_{T}\!\left(  a_{3},b_{3}\right)
\right)  $} & $6.0901$ & $6.0769$ &  & \\\cline{1-3}%
\end{tabular}
\end{center}
\end{table}

\vfill
\pagebreak%

\begin{landscape}%

\section{Table of Polynomials}

\begin{table}[h]
\caption{Admissible Change of Variables for Lemma \ref{ap:lemmaFreytors}}%
\label{ta:bestabclemfreyad}
\begin{center}%
\begin{tabular}
[c]{|c||c|c|c|c|c|}\hline
$T$ & $u_{T}$ & $r_{T}$ & $s_{T}$ & $w_{T}$ & $t_{T}$\\\hline\hline
$C_{2}\times C_{2}$ & $a^{2}$ & $0$ & $0$ & $0$ & $\frac{b}{a}$\\\hline
$C_{2}\times C_{4}$ & $2\left(  a-b\right)  ^{2}$ & $-2ab\left(  a-b\right)
^{2}$ & $\left(  a-b\right)  ^{2}$ & $-2ab\left(  a-b\right)  ^{2}\left(
a^{2}+b^{2}\right)  $ & $\frac{b}{a}$\\\hline
$C_{2}\times C_{6}$ & $9a^{2}-b^{2}$ & $-4a^{2}\left(  a+b\right)  \left(
-3a+b\right)  $ & $5a^{2}-b^{2}$ & $36a^{6}-40a^{4}b^{2}+4a^{2}b^{4}$ &
$\frac{9a+b}{a+b}$\\\hline
$C_{2}\times C_{8}$ & $\frac{1}{2a\left(  a+b\right)  \left(  b^{2}%
-2ab-a^{2}\right)  }$ & $\frac{ab\left(  a^{2}+b^{2}\right)  }{\left(
a+b\right)  ^{2}\left(  b^{2}-2ab-a^{2}\right)  }$ & $\frac{a^{4}%
+4a^{3}b-b^{4}}{2a\left(  a+b\right)  \left(  b^{2}-2ab-a^{2}\right)  }$ &
$\frac{ab^{2}\left(  a^{2}+b^{2}\right)  ^{2}}{\left(  a+b\right)  ^{3}\left(
b^{2}-2ab-a^{2}\right)  ^{2}}$ & $\frac{a}{2\left(  b-a\right)  }$\\\hline
\end{tabular}
\end{center}
\end{table}

\begin{table}[h]
\caption{Polynomials and Rational Functions}%
\label{ta:ABCtriplesPaper}
%
\par
\begin{center}
\noindent\resizebox{\linewidth}{!}{\begin{tabular}
[c]{|c||c|c|c|c|}\hline
$T$ & $C_{2}\times C_{2}$ & $C_{2}\times C_{4}$ & $C_{2}\times C_{6}$ &
$C_{2}\times C_{8}$\\\hline\hline
$\mathfrak{A}_{T}$ & $8ab\!\left(  a^{2}+b^{2}\right)  $ & $\left(  2ab\right)  ^{2}$ &
$16a^{3}b$ & $\left(  2ab\right)  ^{4}$\\\hline
$\mathfrak{B}_{T}$ & $\left(  a-b\right)  ^{4}$ & $\left(  a^{2}-b^{2}\right)  ^{2}$ &
$\left(  a+b\right)  ^{3}\!\left(  b-3a\right)  $ & $\left(  a^{4}-6a^{2}b^{2}+b^{4}\right)  \!\left(  a^{2}+b^{2}\right)  ^{2}$\\\hline
$\mathfrak{C}_{T}$ & $\left(  a+b\right)  ^{4}$ & $\left(  a^{2}+b^{2}\right)  ^{2}$ &
$\left(  3a+b\right)  \!\left(  b-a\right)  ^{3}$ & $\left(  a^{2}-b^{2}\right)  ^{4}$\\\hline
$\mathfrak{D}_{T}$ & $b^{4}-a^{4}$ & $b^{4}-a^{4}$ & $\left(  b^{2}-a^{2}\right)  \!\left(
b^{2}-9a^{2}\right)  $ & $\left(  a^{4}-6a^{2}b^{2}+b^{4}\right)  \!\left(
b^{4}-a^{4}\right)  $\\\hline
$\mathfrak{A}_{T}^{r}  $ & $ab\!\left(  a^{2}+b^{2}\right)  $
& $ab$ & $ab$ & $ab$\\\hline
$\mathfrak{B}_{T}^{r}  $ & $\left(  a-b\right)  $ &
$a^{2}-b^{2}$ & $\left(  a+b\right)  \!\left(  b-3a\right)  $ & $\left(
a^{4}-6a^{2}b^{2}+b^{4}\right)  \!\left(  a^{2}+b^{2}\right)  $\\\hline
$\mathfrak{C}_{T}^{r}  $ & $a+b$ & $a^{2}+b^{2}$ & $\left(
3a+b\right)  \!\left(  b-a\right)  $ & $a^{2}-b^{2}$\\\hline
$f_{T}$ & $\frac{\left(  1-t\right)  ^{4}}{8t\!\left(  1+t^{2}\right)  }-t$ &
$\frac{\left(  1-t^{2}\right)  ^{2}}{\left(  2t\right)  ^{2}}-t$ &
$\frac{\left(  1+t\right)  ^{3}\!\left(  t-3\right)  }{16t}-t$ &
$\frac{\left(  1-6t^{2}+t^{4}\right)  \!\left(  1+t^{2}\right)  ^{2}}{\left(
2t\right)  ^{4}}-t$\\\hline
$g_{T}$ & $4t^{3}+6t^{2}+4t+2$ & $2t^{2}+2$ & $4t^{2}+8t-12$ & $2t^{6}+6t^{4}-10t^{2}+2$\\\hline
$\theta_{T}$ & $1$ & $1$ & $4.87517$ & $3.17374$\\\hline
$U_{T}$ & $\begin{array}
[c]{c}5a^{3}r+20a^{2}br+29ab^{2}r+\\
16b^{3}r+16a^{3}s+29a^{2}bs+\\
20ab^{2}s+5b^{3}s
\end{array}
$ & $\begin{array}
[c]{c}a^{2}r+2b^{2}r+\\
2a^{2}s+b^{2}s\\
\end{array}
$ & $\begin{array}
[c]{c}-54a^{3}r+144a^{2}br-117ab^{2}r+\\
24b^{3}r-8a^{3}s+\\
6a^{2}bs-b^{3}s
\end{array}
$ & $\begin{array}
[c]{c}4a^{6}r-15a^{4}b^{2}r+20a^{2}b^{4}r-\\
10b^{6}r-10a^{6}s+\\
20a^{4}b^{2}s-15a^{2}b^{4}s+4b^{6}s
\end{array}
$\\\hline
$V_{T}$ & $\begin{array}
[c]{c}-5a^{3}r+20a^{2}br-29ab^{2}r+\\
16b^{3}r+16a^{3}s-\\
29a^{2}bs+20ab^{2}s-5b^{3}s
\end{array}
$ & $\begin{array}
[c]{c}-a^{2}r+2b^{2}r+\\
2a^{2}s-b^{2}s\\
\end{array}
$ & $\begin{array}
[c]{c}54a^{3}r+144a^{2}br+\\
117ab^{2}r+24b^{3}r-\\
8a^{3}s-6a^{2}bs+b^{3}s
\end{array}
$ & $\begin{array}
[c]{c}-4a^{6}r+15a^{4}b^{2}r+44a^{2}b^{4}r+\\
26b^{6}r+26a^{6}s+\\
44a^{4}b^{2}s+15a^{2}b^{4}s-4b^{6}s
\end{array}
$\\\hline
$W_{T}$ & $32\left(  ra^{7}+sb^{7}\right)  $ & $4\left(  ra^{6}+sb^{6}\right)  $ &
$48\left(  ra^{7}+sb^{7}\right)  $ & $16\left(  ra^{14}+sb^{14}\right)
$\\\hline
\end{tabular}}
\end{center}
\end{table}%

\end{landscape}%

\pagebreak

\bibliographystyle{amsalpha}
\bibliography{bibliography}

\end{document}